\documentclass{article}
\usepackage{etex}
\usepackage{amscd,amsmath,amsfonts,amsthm,amssymb}
\usepackage{mathrsfs}
\usepackage[all]{xy}
\usepackage{hyperref}
\usepackage{graphicx}
\usepackage{comment}
\usepackage{subfig}
\usepackage{epsf}
\usepackage{epsfig}
\usepackage{xspace}
\usepackage{graphicx}
\usepackage[active]{srcltx}
\usepackage[table]{xcolor}
\usepackage{epstopdf} 
\usepackage{tikz}
\usetikzlibrary{arrows,shapes}
\usepackage{mathtools}
\usepackage[font=footnotesize]{caption}
\usepackage{tabularx,ragged2e,booktabs,caption}
\usepackage{latexsym}
\usepackage{subfig}
\usepackage{datetime}

\newtheorem{theorem}{Theorem}[section]

\newtheorem{lemma}[theorem]{Lemma}
\newtheorem{proposition}[theorem]{Proposition}

\theoremstyle{definition}
\newtheorem{definition}[theorem]{Definition}

\theoremstyle{remark}
\newtheorem{remark}{Remark}[theorem]

\numberwithin{equation}{section}

\newcommand{\N}{\mathbb N}

\newcommand{\SG}{\mathcal G}
\newcommand{\CC}{\mathrm{C}}
\newcommand{\PC}{\mathcal{P}}
\newcommand{\T}{\mathcal{T}_f}
\newcommand{\Tz}{\mathcal{T}_f^0}

\newcommand{\HH}{\mathrm{H}}
\newcommand{\HHS}{\mathrm{H}^\mathrm{S}}
\newcommand{\ZZ}{\mathrm{Z}}
\newcommand{\BB}{\mathrm{B}}

\newcommand{\gm}{ k_1}

\newcommand{\Ord}{\mathcal{O}}

\newcommand{\aaa}{\mathcal{A}}
\newcommand{\sss}{\mathcal{S}}
\newcommand{\cl}{\mathrm{Cl}}

\begin{document}

\title{$P$-persistent homology of finite topological spaces}
\author{F. Vaccarino, A. Patania, G. Petri}

\maketitle

\begin{abstract}

Let $P$ be a finite partially ordered set, briefly a finite \textit{poset}. We will show that for any $P$-persistent object $X$ in the category of finite topological spaces, there is a 
$P-$ weighted graph, whose clique complex has the same $P$-persistent homology as $X$.

\end{abstract}

\section{Introduction}

The study of topological spaces and the related computational methods are receiving an unprecedented attention from fields as diverse as biology and social sciences \cite{carlsson 2009}, \cite{chung}, \cite{nicolau}, \cite{topologicalstrata} and \cite{homscaffold}. 

The original motivation of this work is to provide a firm mathematical background for the results obtained in \cite{topologicalstrata,homscaffold}, where the authors defined a filtration of a weighted network directly in terms of edge weights. 
The rationale behind this was the observation that embedding a network into a metric space generally obfuscates most of its interesting structures \cite{eccs}, which become instead evident when one focuses on the weighted connectivity structures without enforcing a metric.

Figure \ref{example_figure} illustrates this through the $H_1$ and $H_2$ persistent diagrams for two different filtrations obtained from a dataset of face-to-face contacts among children in an elementary school (see the Sociopatterns project \cite{sociopatterns} for details). 

The \textit{metrical filtration} is obtained in the standard way: given a metric (weighted shortest path in this case), one constructs a sequence of Rips-Vietoris complexes by studying the change in the overlap of $\epsilon$-neighbourhoods of vertices while varying their radius $\epsilon$ (Figure \ref{example_figure} right). 
The non metrical one relies instead on associating { clique} complexes to a series of binary networks obtained from a progressively descending thresholding on edge weights (Figure \ref{example_figure} left). 
The difference between the diagrams of the two filtration is evident: in the first case, most of the generators have short persistent and are thus distributed along the diagonal; in the second generators display a range of persistents, including some very large ones that signal the presence of interesting heterogeneities in the network structure. 
\begin{figure*}[!pbt]
\centering 
\includegraphics[width=.86\textwidth]{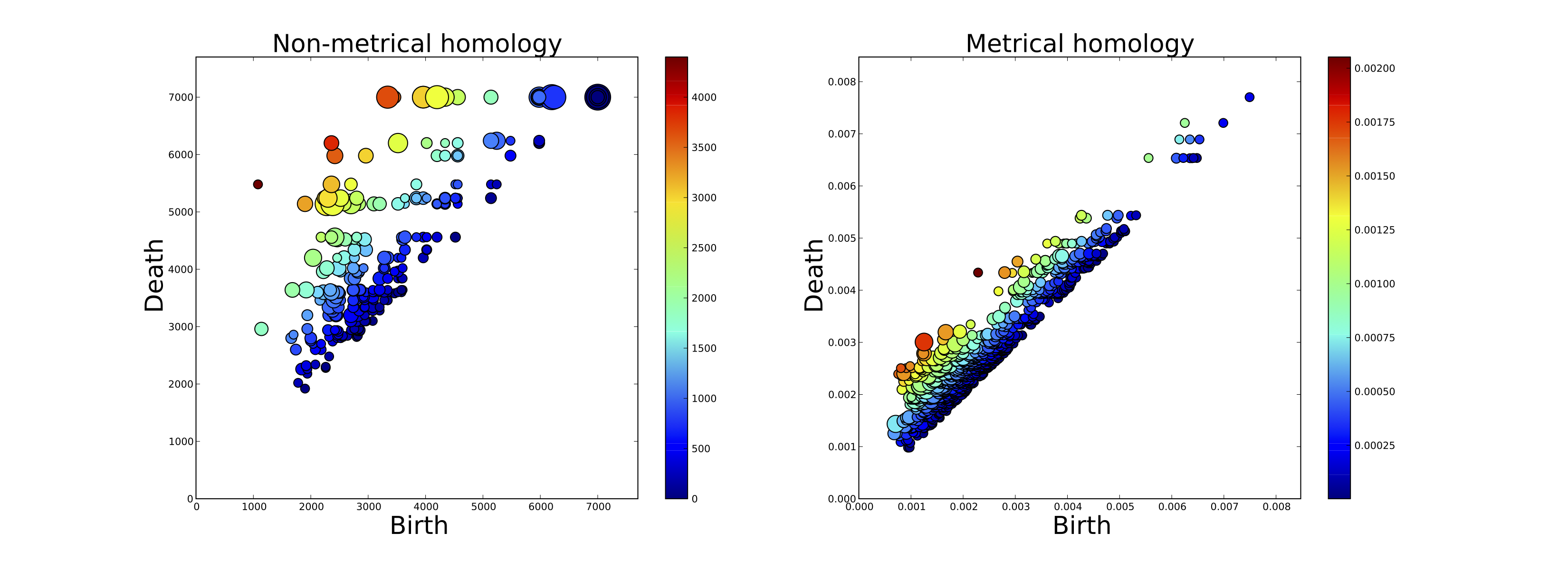}
\includegraphics[width=.86\textwidth]{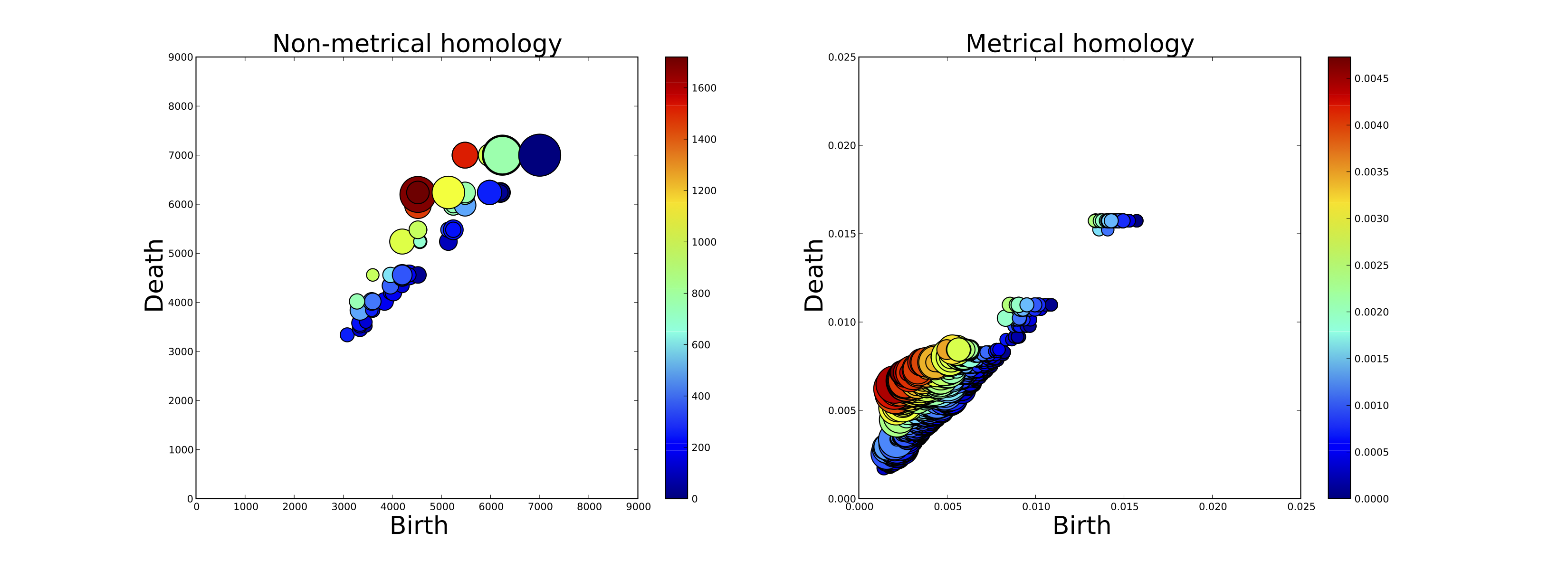}
\caption{Comparison of results obtained for $H_1$ ({\it top row}) and $H_2$ ({\it bottom row}) on the Sociopatterns network with weighted graph filtration ({\it left}) and metrical filtration ({\it right} as described in the main text. The colour of dots represents the corresponding generator's persistent, while the size is proportional to number of simplices composing a chain representative of the generator's  class (we choose here as representative the standard output of javaplex \cite{javaplex}).}\label{example_figure}
\end{figure*}

More interestingly, we will see that the attempt at providing a mathematical formalisation of the approach used in \cite{topologicalstrata} yields a much more general result regarding $P-$ persistent homology. \\

In order to obtain this result, we need to introduce a number of notions. 
In the next section (Sec. \ref{pers}) we briefly define the notion of $P-$persistence, which is our main object of study. Subsequently we lay down the vocabulary needed, introducing finite topological spaces and their equivalence to partially ordered sets, the relation between simplicial complexes, their associated order complexes and graphs. 
Finally, in Section \ref{hom} we introduce the homology and prove the theorem in Section \ref{onegraph}, highlighting how the metrical case is only a specific and constrained case of the possible network weightings.

\section{$P$-persistence}\label{pers}

Let us recall the definition of poset.
\begin{definition}
A partially ordered set, briefly a \textit{poset}, is a pair $P=(P,\leq)$, where $P$ is a set and $\leq$ is an order relation on it, i.e. a reflexive, antisymmetric, and transitive relation on $P.$\\
Posets form a category, denoted by $\PC$, where morphisms are the order preserving functions. In this paper we will consider only finite posets.
\end{definition}

\begin{remark}\label{poscat}
Every poset $P$ is a category on its own, where the objects are the elements of $P$, and there is a (unique) morphism $x\to y$ if and only if $x\leq y$, for all $x,y\in P.$ 
\end{remark}
\medskip
Following Section 2.3 in \cite{carlsson 2009}, a $P$-persistent object is given by the following definition.

\begin{definition} 
Let $P$ be a poset and $\aaa$ an arbitrary category. A $P$\emph{-persistent object} in $\aaa$ is a functor $\varphi:P\to \aaa.$
$P$-persistent objects in $\aaa$ with their natural transformation form a category, which we will denote, as usual, by ${{ \aaa}^P}.$
Given two categories $\mathcal{A}$ and $\mathcal{B},$ to any functor $\phi:\mathcal{A}\to \mathcal{B}$ it corresponds a functor $\mathcal{A}^P\to \mathcal{B}^P$.
It is given by $\varphi\in\mathcal{A}^P\mapsto \phi\circ\varphi$. It will be denoted by $\phi^P.$
\end{definition}

As said in the introduction, we are interested in studying persistent objects in the category of finite topological spaces and in a suitable category of graphs. Let us then introduce the reader to the categories we are interested in and some functors relating them.
\subsection{Basic categories and functors}

We want to start with a few considerations on finite topological spaces i.e. topological spaces with a finite number of elements, which we can imagine to be given as some sampling taken from a dataset.  
Finiteness is not a constraint for our purposes, since every application will have a finite data space, although possibly very large. \\

Finite topological spaces form a subcategory, denoted by $\T$, of the category $\mathcal{T}$ of 
 topological spaces and continuous maps. 
 
 A $\mathrm{T_0}$ space is a topological space such that for any two different points $x$ and $y$ there is an open set which contains one of these points and not the other. Two such points will be called topologically distinguishable. It is clear that this property is highly desirable in order to be able to extract meaningful information from a topological space.
 
In this paper we will denote by $\Tz$ the category of finite $\mathrm{T_0}-$spaces.

From here on when we write topological space we will intend finite topological space if not elsewhere stated.

\subsubsection{Kolmogorov quotient}
It may happen that a space we are working with is not $\mathrm{T_0},$ but this difficulty is easy to overcome as it is shown by the following, which is well known, see e.g. \cite{Barmak} Proposition 1.3.1.
\begin{proposition}\label{kolmo}
Let $X$ be a finite space not $\mathrm{T_0}$. Let $X/\sim$ be the Kolmogorov quotient of $X$ defined by $x\sim y$ if it does not exists an open set which contains one of these points and not the other. Then, $X/\sim$ is $\mathrm{T_0}$ and the quotient map $q:X\to x_0$ is a homotopy equivalence.\\
The Kolmogorov quotient $X\to X/\sim$ induces a functor from the category of topological spaces to the category of $\mathrm{T_0}-$spaces.
\end{proposition}

Since homology is defined up to weak homotopy equivalence, the Kolmogorov quotient allows us to restrict our analysis from general topological spaces to $\mathrm{T_0}-$spaces without any loss of information. 

\subsubsection{Finite $\mathrm{T_0}-$spaces are posets}

\begin{theorem}\label{tzcongpc}
There is an isomorphism of categories:
\[\Tz\cong\PC\]
\end{theorem}
\begin{proof}
Let $X\in\Tz$, for $x\in X$ let $U_x$ be the intersection of all the closed sets in X that contain $x$. Then we can give in $X$ an order relation in the following way:
\begin{equation}\label{rel}
x\leq y \leftrightarrow U_x\subseteq U_y
\end{equation}
Since $X$ is $\mathrm{T_0}$ this relation is a partial order. In this way we have a correspondence $X\mapsto (X,\le)$ which induces a functor $\Tz\to\PC$.\\
On the other end, a poset $P\in\PC$ is also a topological space via the \textit{Alexandrov topology}. In this topology the closed sets are the \textit{lower sets}:
$\Gamma\subset P $ such that $\forall x,y\in P$ with $ x\in\Gamma$ and $y\leq x$ implies that $y\in \Gamma$.
A poset endowed with this topology satisfies the $\mathrm{T_0}$ condition. 
The assignment of this topology on $P$ induces a functor $\PC\to\Tz$ which is left and right inverse of the previous one, that is:
\begin{equation}
\Tz\cong\PC
\end{equation}
We refer the reader to Chapter 1 \cite{Barmak} for details.
\end{proof}

\begin{remark}
(i) A preordered set is a set endowed with a binary relation which is reflexive and transitive. A poset is a special kind of preorder obtained by requiring the relation to be antisymmetric.
Preordered sets form a category $Preord$ the same way as posets. 
The isomorphism in Theorem \ref{tzcongpc} can be extended to an isomorphism $\T\cong Preord_f,$ where the latter is the category of finite preordered sets.\\
(ii) Let $X\in\T$  be a $T_1-$space, i.e. for all $x,y\in X$ there exist two open sets $A,B\subseteq X$ such that $x\in A$, $y\in B$ and $y\notin A$, $x\notin B$.
Then $X$ has the discrete topology and the process just described is not very informative because in this case the poset $(X,\leq)$ is discrete. \\
We will see how to deal with finite metric spaces, that are $T_1$, in the last section of this paper.
\end{remark}

From now on we will identify any $X\in \T$ with the poset associated to its Kolmogorov quotient i.e., by abuse of notation, we will write 
\[X=(X,\le)=(X/\sim, \le)\] 
where $\leq$ is the order relation given in (\ref{rel}).

\subsubsection{Simplicial Complexes}
The basic idea of a simplicial complex is that of gluing together, in a coherent way, points, lines, triangles, tetrahedra, and higher dimensional equivalents. We will give now a more formal definition.
\begin{definition}
An (abstract) \emph{simplicial complex} is a non empty family $\Sigma$ of finite subsets, called faces, of a vertex set $V$ such that $\sigma\subset \tau \in \Sigma$ implies that $\tau\in \Sigma.$
\end{definition}

We assume that the vertex set is finite and totally ordered. A face of $n+1$ vertices is called $n-$face, denoted by $[p_0,\ldots, p_n]$, and $n$ is its dimension. We set, as usual, the dimension of the empty set as -1 , following Section 2.1 in \cite{Kozlov}.\\
A $0-$face is a vertex, a $1-$face is an edge, a $2-$face is a full triangle, a $3-$face is a full tetrahedron, etc.\\
The \textit{dimension} of a simplicial complex the highest dimension of the faces in the complex. We call \textit{vertex set} $V$ of $\Sigma$ the union of the one point elements  of $\Sigma$.\\
For a simplicial complex $\Sigma$ and a non negative integer $k$ we denote by $\Sigma_k$ its $k-skeleton$ which is defined as
\begin{equation}\label{ale}
\Sigma_k:=\{\sigma\in\Sigma\, : \, \dim\sigma\leq k\}
\end{equation}
The $k-$skeleton of $\Sigma$ is a simplicial complex in the obvious way.\\
Simplicial complexes form a category, $\sss$, where a morphism of simplicial complex is called \textit{simplicial map} and is given by a map on vertices such that the image of a face is again a face.\\

We are going to remind some well known relations between simplicial complexes, topological spaces and posets.

\begin{proposition}
There exists a functor $\Ord:\mathcal{P}\to\sss$ which associates to every poset $P$ a simplicial complex, called the order complex.
\end{proposition}

\begin{proof}
For every $P\in\PC$ we can construct a simplicial complex as follows:
\begin{description}
\item $[x_0,\dots,x_k]\in\mathcal{O}(P)$ if and only if $x_0< x_1<\dots< x_k$, for all $x_j\in P$.
\end{description}
$\Ord(P)$ is called the \textit{order complex of} $P$.
\end{proof}

\begin{definition}\label{alexS}
Every simplicial complex can be made into a topological space by considering it a poset, i.e. $\Gamma\subseteq\Sigma$ is closed if and only if $\Gamma$ is a simplicial complex. 
This gives a functor $\pi:\sss\to\PC\simeq\Tz$ by $\pi(\Sigma)=(\Sigma, \subseteq)$ the poset with elements the simplices in $\Sigma$ and as partial order the inclusion of simplices.\\
Given a simplicial complex $\Sigma$ we write $\mathcal{O}(\Sigma):=\mathcal{O}(\pi(\Sigma))$. The simplicial complex $\mathcal{O}(\Sigma)$ is called the \textit{barycentric subdivision} of $\Sigma$.\\
 By abuse of notation we will also write $\mathcal{O}(X):=\mathcal{O}(X/\sim,\leq)$ for all $X\in\T$ via the isomorphism in Theorem \ref{tzcongpc}.
 \end{definition}

It is well known that $\Sigma$ and $\mathcal{O}(\Sigma),$ endowed with the Alexandrov topology, are weakly homotopy equivalent. We refer the interested reader to \cite{Barmak} for further details.

\subsubsection{Graphs}
\begin{definition}
A \textit{reflexive graph} is a pair $G=(V,E)$, where $V$ is a finite set whose elements are called vertices and  $E=\Delta_{V\times V}\cup E'$ with $E'\subseteq {V\choose 2}$, i.e. is a graph which has an edge (called \textit{self-loop}) $(v,v)$ for every vertex $v\in V$.
Equivalently, reflexive graphs can be seen as one dimensional simplicial complexes, identifying self-loops and vertices with $0$-simplices and edges with $1$-simplices. We will denote by $\SG$ the category with objects reflexive graphs and morphisms the simplicial maps defined via the given identification with one dimensional simplicial complexes. 
\end{definition}

It should be clear that $\SG$ is isomorphic to the full subcategory of $\sss$ whose objects are the one dimensional simplicial complexes.\\

\begin{remark}
It is useful to notice that, since graphs in $\mathcal{G}$ are defined as $G=(V,E)$, then the null graph $G_\emptyset=(\emptyset,\emptyset)$ is an object in $\mathcal{G}$.
\end{remark}

\begin{definition}
A \textit{{ clique}} in a graph $G=(V,E)$ is a complete subgraph of $G$ i.e. is a subgraph $G'=(V',E')$ with $V'\subset V$ and $E'\subset E$ such that $E'= {V'\choose 2}$.\\
Given  a graph $G\in\SG$ there is a covariant functor, $\cl:\SG\to \sss$, called the \textit{{ clique} functor} given by $[v_0,\dots,v_k]\in \cl(G)$ if and only if $(v_i,v_j)\in E$ for all $0\leq i\neq j\leq k.$. 
\end{definition}
Note that this is well defined because $(v,v)\mapsto [v],$ for all $v\in V.$
\begin{definition}
Given a simplicial complex $\Sigma$ there is a functor $ k_1:\sss \to \SG$ where $\gm(\Sigma)$ is the (reflexive) graph corresponding to the $1-$skeleton of $\Sigma$.
\end{definition}
\begin{remark}
In general $\Sigma\neq\cl( k_1(\Sigma))$. For example if we consider $\Sigma=\{[a],[b],[c],[a,b],[a,c],[b,c]\}$, this is a simplicial complex but $\cl( k_1(\Sigma))=\Sigma\cup\{[a,b,c]\}$. 
\end{remark}
\begin{definition}
A simplicial complex $\Sigma$ is called a \textit{{ flag}} (or a \textit{{ clique}} complex) if $\Sigma=\cl( k_1(\Sigma))$. { flag} complexes form a subcategory of $\sss$ denoted by $\mathcal{F}$.
\end{definition}
\begin{remark}\label{**}
It is easy to see that $\mathcal{O} (X)$ is a { flag} complex for all $X\in\T.$\\
In particular this implies that, for all $\Sigma\in\sss,$ the barycentric subdivision $\Ord(\Sigma)$ is a { flag} complex.\\
\end{remark}

\begin{proposition}
The functors $\cl:\mathcal{G}\to\mathcal{F}$ and ${\gm}_{|\mathcal{F}}:\mathcal{F}\to \mathcal{G}$ give an isomorphism $\mathcal{G}\simeq \mathcal{F}.$
\end{proposition}
\begin{proof}
Obvious.
\end{proof}

To summarize:
\begin{equation}\label{dicafu}
\xymatrix{\sss \ar [r]^(.35){\pi} \ar@/_1pc/[rr]_{\Ord\circ\pi}
&\PC\simeq\Tz \ar [r]^(.60){\Ord} 
&\ar @{} [r] |{\cong}
\mathcal{F} \ar @/_1pc/ [r]_\gm &\mathcal{G} \ar @/_1pc/ [l] _\cl\\
&&&
}
\end{equation}

\subsection{$P$-weighted graphs}
\begin{definition}
Let $P\in\PC$ be a poset and $G=(V,E)\in \SG$ a graph, let us denote by $G\in \sss$ the corresponding one dimensional simplicial complex. A $P-weighted\,\, graph$ is a pair $(G,\omega),$ where $\omega:(G, \subseteq)\to P$ is a morphism of posets i.e. a function $G\to P,$ which is continuous in the Alexandrov topology.\\
\end{definition}
\begin{definition}
We define as $\SG_P$ the category of $P-$weighted graphs, having objects $P-$weighted graphs and  whose morphisms $\alpha:(G,\omega)\to (H,\theta)$ are induced by a simplicial map $\rho:G\to H$, such that $\alpha({G}_v)\subseteq {H}_v$, where for any $v\in P$, $G_v=\{x\in G \,|\, \omega(x)\le v\}$.
\end{definition}

\section{Main results: equivalences and adjunctions}\label{onegraphsec}
We have the following.
\begin{proposition}
For all $P\in\PC,$ there is a functor $\Phi_P:\mathcal{G}_P\to \mathcal{G}^P$.
\end{proposition}
\begin{proof}
Let $(G,\omega)\in \mathcal{G}_P$. From the definition of $G_v$ we have that $G_u\subseteq G_v$ for every $u\leq v$.\\
We can associate to $(G,\omega)\in\mathcal{G}_P$ a $P-$persistent object in $\varphi_G\in \mathcal{G}^P$, namely $\varphi_G(v)=G_v$ with the inclusions maps $\varphi_G(u\le v):G_u\hookrightarrow G_v$ for all $v \in P$, $u \in P_v$.
It is easy to check that the correspondence $(G,\omega)\to \varphi_G$ is natural in $G$. Therefore $\Phi_P$ is a functor between the two categories.
\end{proof}

\begin{proposition}\label{psifunctor}
For all $P\in\PC$ there exists a functor $\Psi_P:\mathcal{G}^P\to\mathcal{G}_P$.
\end{proposition}
\begin{proof}
Choose $\varphi\in\mathcal{G}^P,$ and, for every $v\in P$, set $\varphi_v:=\varphi(v)\in\mathcal{G}.$\\
Let  $\omega^{\varphi}:\coprod_{v\in P}\varphi_v\to P$ be given by $\omega^{\varphi}|_{\varphi_v}=v$. It is easy to check that the correspondence $\varphi\mapsto (\coprod_{v\in P}\varphi_v, \omega^{\varphi})$ is natural in $\varphi$, thus giving a functor $\Psi_P: \mathcal{G}^P\to\mathcal{G}_P$. 
\end{proof}

\begin{definition}
Let $\bar{\mathcal{G}}_P$ be the subcategory of $\mathcal{G}_P$ with the same objects, and morphisms the maps $\alpha:(C,\omega)\to (D,\omega')$ such that for every $x\in G$, $\omega'(\alpha(x)) = \omega(x)$. We set $\bar{\Phi}_P$ as the restriction of $\Phi_P$ to $\bar{\mathcal{G}}_P$.
\end{definition}
\begin{remark}
It is useful to notice that, actually, $\Psi_P:\mathcal{G}^P\to\bar{\mathcal{G}}_P$. Since $\Psi_P(\varphi)\in Ob(\mathcal{G}_P)=Ob(\bar{\mathcal{G}}_P)$ for all $\varphi\in \mathcal{G}^P$, we just show that $\Psi_P(\mu)$ preserves weights for every $\mu:\varphi\to\tau \in \mathcal{G}^P$. Indeed from the definition of the weights $\omega^{\varphi}, \omega^{\tau}$ we have that $(\omega^{\varphi})^{-1}(u)=\varphi_u$ then $\Psi_P(\mu)(\varphi_u) \subseteq\tau_u=(\omega^{\tau})^{-1}(u)$. 
\end{remark}

\begin{definition}
Let $\mathcal{G}_\iota^P$ be the subcategory of $\mathcal{G}^P$ whose objects are $\varphi\in\mathcal{G}^P$ such that the morphisms $\varphi(u\le v):\varphi(u)\to\varphi(v)$ are inclusions.\\
We set $\Psi_P^\iota$ as the restriction of $\Psi_P$ to $\mathcal{G}^P_\iota$.
\end{definition} 

Following Carlsson \cite{(Multi2)} we introduce the concept of \textit{one critical} $P-$persistent object.
\begin{definition}
Let $\varphi\in\mathcal{G}_\iota^P$. $\varphi$ is said to be one-critical if for all $v\in P$, for all $(x,y)\in E_{\varphi(a)}$
\begin{equation}\label{min}
\exists!\; m_{xy}=\min\{u\in P| \varphi_{uv}(x,y)=(x,y)\}
\end{equation}
The one-critical $P$-persistent objects form a subcategory of $\mathcal{G}_\iota^P$, which will be denoted by $\SG^P_{\mathbf{1}}$. We set $\Psi^{\mathbf{1}}_P$ as the restriction of $\Phi_P$ to $\mathcal{G}_\mathbf{1}^P$.
\end{definition}

\begin{remark}
It is useful to notice that $\Phi_P:\mathcal{G}_P\to\mathcal{G}^P_\mathbf{1}$. Indeed consider $(G,\omega)\in \mathcal{G}_P.$ By definition of $\Phi_P$, it is clear that $\varphi_G\in\mathcal{G}^P_\mathbf{1},$ with $m_{xy}=\omega(x,y).$
\end{remark}

\begin{theorem}\label{isocat}
The categories $\mathcal{G}_\mathbf{1}^P$ and $\bar{\SG}_P$ are equivalent.
\end{theorem}


\begin{proof} It is a well known fact in category theory that a functor is an equivalence if and only if it is full, faithful and essentially surjective. Then to prove the equivalence of category we need to verify that $\Phi_P$ has these three properties.\\

Consider $(G,\omega),(H,\theta)\in{\bar{\SG}_P}$, and 
$\alpha\in\hom_{\bar{\SG}_P}((G,\omega),(H,\theta)).$ The functor $\Phi_P$ is essentially surjective if it is surjective on objects up to isomorphism. 
Let $\varphi\in{\SG^P_\mathbf{1}}$, then we can construct
 $(G,\omega)\in\bar{\SG}_P$ by $G:=\bigcup_{a\in P}\varphi(a)$ and $\omega((x,y))=m_{xy}$ (see \ref{min}).
 It follows that $\Phi_P((G,\omega))$ is such that, for all $u\in P,$ one has $\varphi_G(u)=\{x\in G|\;\omega(x)\le u\}=\bigcup_{a\in P; a\leq u}\varphi(a)\cong\varphi(u)$ by definition of $\varphi.$\\
 
The functor $\Phi_P$ is full if the map 
\[\Phi_P((G,\omega),(H,\theta)):\hom_{\bar{\SG}_P}((G,\omega),(H,\theta))\to\hom_{\mathcal{G}_\mathbf{1}^P}(\varphi_G,\varphi_H)\]
 is surjective for all 
 $(G,\omega),(H,\theta)\in\bar{\SG}_P$.\\
Consider a morphism $\rho:\varphi_G\to\varphi_H$ in $\mathcal{G}_\mathbf{1}^P$. 
Let $\alpha:E_G\to E_H$ be given by $\alpha((x,y))=\rho_{\omega(x,y)}(x,y),$ for every $(x,y)\in E_G.$
Then $\alpha$ is a morphism $(G,\omega)\to (H,\theta)$ in $\bar{\SG}_P,$ because, since 
$\rho_{\omega(x,y)}:\varphi_G(\omega(x,y))\to\varphi_H(\omega(x,y))$,
 we have that $\theta(\alpha((x,y)))= \omega(x,y)$.  It is clear that $\Phi_P(\alpha)$ is $\rho$ by the definitions of $\Phi_P$ and $\alpha.$
 
As last step, we prove that $\Phi_P$ is said faithful, i.e.  that the map 
\[\Phi_P((G,\omega),(H,\theta)):\hom_{\bar{\SG}_P}((G,\omega),(H,\theta))\to\hom_{\mathcal{G}_\mathbf{1}^P}(\varphi_G,\varphi_H)\] 
is injective for all $(G,\omega),(H,\theta)\in\bar{\SG}_P$.\\
Consider $\alpha,\beta\in\hom_{\bar{\SG}_P}((G,\omega),(H,\theta))$ such that $\Phi(\alpha)=\Phi(\beta).$
This means that $\Phi(\alpha)_v=\Phi(\beta)_v$ for all $v\in P$, but this implies that $\alpha|_{G_v}=\beta|_{G_v}$ for all $v\in P$, then $\alpha=\beta$.
\end{proof}

\subsection{Adjunctions}
Beside the equivalence in Th.\ref{isocat}, there are also some results on the relationships between the other categories involved. 
\begin{theorem}
$\bar{\Phi}_P$ is left adjoint of $\Psi_P$, that is \[\hom_{\bar{\mathcal{G}}_P}((X,\omega),\Psi_P(\varphi))\cong\hom_{\mathcal{G}^P}(\bar{\Phi}_P((X,\omega)), \varphi)\]
\end{theorem}
\begin{proof}
Let $\pi:(X,\omega)\to (\coprod_P X_u,\omega^{\varphi_X})$ given by $\pi(x)=(x,\omega(x))\in X_{\omega(x)}$, for every $x\in (X,\omega)$. This map is well defined and is actually a morphism in the category $\bar{\mathcal{G}}_P$ since $\omega^{\varphi_X}(\pi(x))=\omega^{\varphi_X}(\,(x,\omega(x))\,)=\omega(x).$\\
To prove the assumption we will show that,  for every $\alpha\in\hom_{\bar{\mathcal{G}}_P}((X,\omega),\Psi_P(\varphi)),$ there is a unique morphism in $\mathcal{G}^P$, $\bar{\alpha}:\bar{\Phi}_P((X,\omega))\to \varphi$ such that the following diagram commutes
\begin{equation}\label{adj1}
\xymatrix{(X,\omega) \ar[r]^(0.35){\pi}\ar @{->} @< 0pt>[d]_\alpha & \Psi_P(\bar{\Phi}_P((X,\omega))) \ar @{->} @< 2pt> [dl]^{\Psi_P(\bar{\alpha})}\\
\Psi_P(\varphi) &
}
\end{equation}

Let $\alpha\in\hom_{\bar{\mathcal{G}}_P}((X,\omega),(\coprod_P \varphi(u),\omega^\varphi))$, then $\alpha$ will be such that $\omega^\varphi(\alpha(x))=\omega(x)$,  for every $x\in (X,\omega)$. By construction of $\omega^\varphi$ we will have that $\alpha(x)\in \varphi(\omega(x))$, and with a little abuse of notation we will write $\alpha: x\mapsto (\alpha(x),\omega(x))$.

Let $\bar{\alpha}:\bar{\Phi}_P((X,\omega))\to \varphi$ be the morphism in $\mathcal{G}^P$ defined through $\bar{\alpha}|_{X_u}:X_u\to\varphi(u)$ with $\bar{\alpha}|_{X_u}(x)=\varphi_{\omega(x)u}(\alpha(x),\omega(x))$, where $\varphi_{\omega(x)u}=\varphi(\omega(x)\le u):\varphi(\omega(x))\to\varphi(u)$. \\

We still have to show that diagram \ref{adj1} commutes, i.e. $\Psi_P(\bar{\alpha})\circ\pi = \alpha$. Let $x$ be an element of $(X,\omega)$ with weight $\omega(x)$, then $\pi(x)=(x,\omega(x))\in X_{\omega(x)}$, so $\bar{\alpha}(\pi(x))=\varphi_{\omega(x)\omega(x)}(\alpha(x),\omega(x))=(\alpha(x),\omega(x))$. Follows that the diagram commutes and this proves the adjuction.\\
%
\end{proof}
There is another adjunction.


\begin{theorem}\label{adjiota}
$\Psi_P^\iota$ is left adjoint of $\Phi_P$.
\end{theorem}

In order to prove this theorem we need some technical lemmata.

\begin{lemma}\label{eps}
There is a natural transformation $\epsilon:\Psi^\iota_P\Phi_P\longrightarrow 1_{\mathcal{G}_P}$.
\end{lemma}
\begin{proof}
Consider $(G,\omega)\in \mathcal{G}_P$, then
\begin{equation}
\xymatrix{
{\mathcal{G}_P} \ar[r]^{\Phi_P} & {\mathcal{G}^P_\iota} \ar[r]^{\Psi_P^\iota} & {\mathcal{G}_P}\\
(G,\omega)\ar @{|->}[r]&\{G_v\}\ar @{|->}[r]&(\coprod_{v\in P}G_v,\omega^{\varphi_G})
}
\end{equation}
Define now $\varepsilon_{(G,\omega)}$ as follows:
\begin{equation}
\begin{array}{rcl}
\varepsilon_{(G,\omega)}:(\coprod_{v\in P}G_v,\omega^{\varphi_G})&\longrightarrow &(G,\omega)\\
(x,u)&\mapsto &x\\
\end{array}
\end{equation}
This map is well defined since $\omega^{\varphi_G}((x,u))=u\geq \omega(x)=\omega(\varepsilon_{(G,\omega)}(x,u))$.\\
Consider now $(F,\tau)$, and $\alpha:(G,\omega)\to(F,\tau)$ in $\mathcal{G}_P$, trivially the following diagram commutes:
\begin{equation}
\xymatrix{
(G,\omega) \ar[rr]^{\Psi_P^\iota\circ\Phi_P} \ar [d]^\alpha&& (\displaystyle\coprod_{v\in P}G_v,\omega^{\varphi_G})\ar[rr]^{\varepsilon_{(G,\omega)}}\ar[d]_{\Psi_P^\iota(\Phi_P(\alpha))}&&(G,\omega)\ar [d]^\alpha\\
(F,\tau) \ar[rr]_{\Psi_P^\iota\circ\Phi_P}&&(\displaystyle\coprod_{v\in P}F_v,\omega^{\varphi_F})\ar[rr]_{\varepsilon_{(F,\tau)}}&&(F,\tau) 
}
\end{equation}
$\varepsilon$ is the natural transformation we were searching for.
\end{proof}

\begin{lemma}\label{eta}
There is a natural trasformation $\eta:\Psi^\iota_P\Phi_P\longrightarrow 1_{\mathcal{G}^P_\iota}$.
\end{lemma}
\begin{proof}

Consider $\varphi\in \mathcal{G}^P_\iota$, $\Phi\circ\Psi(\varphi)=(\coprod_{u\le v} \phi(u),\subseteq)$.
\begin{equation}
\xymatrix{
{\mathcal{G}^P_\iota} \ar[r]^{\Psi_P^\iota} & {\mathcal{G}_P} \ar[r]^{\Phi_P} & {\mathcal{G}^P_\iota}\\
\{\varphi(v),\subseteq\}_{v\in P}\ar @{|->}[r]&(\coprod_{v\in P}\varphi(v),\omega^\varphi)\ar @{|->}[r]&\{\coprod_{u\le v}\varphi(u),\subseteq\}
}
\end{equation}
Define now $\eta_\varphi$ as follows:
\begin{equation}
\begin{array}{rcl}
\eta_\varphi:
\varphi(v)&\to&\coprod_{u\le v}\varphi(u)\\
x&\mapsto &(x,v)
\end{array}
\end{equation}
Consider now $\theta$, and $\alpha:\varphi\to\theta$ in $\mathcal{G}^P_\iota$, the following diagram commutes:
\begin{equation}
\xymatrix{
\varphi(v) \ar[r]^{\eta_\varphi} \ar [d]^{\alpha_v}& \displaystyle\coprod_{u\le v}\varphi(u) \ar [d]^{\coprod\alpha_u}\\
\theta(v) \ar[r]^{\eta_\theta} & \displaystyle\coprod_{u\le v}\theta(u)
}
\end{equation}
where for every $(x,w)\in \coprod_{u\le v}\varphi(u)$, $\coprod\alpha_u((x,w))=(\alpha(x),w)$.\\
$\eta$ is the natural transformation we were searching for.
\end{proof} 

\begin{proof}[Proof of Theorem \ref{adjiota}]
We will prove the unit-counit adjunction, with $\varepsilon$ and $\eta$ the natural transformations defined in Lemma \ref{eps}, and \ref{eta}.\\

To prove the adjunction we verify that the following compositions are the identity transformation of the respective categories.
\begin{equation}
\xymatrix{{\Phi_P} \ar '[d]
'[drr] ^{id_{\mathcal{G}_P}}
[rr] 
\ar[r]^(0.35){\eta\Phi}  
& \Phi_P\Psi_P^\iota\Phi_P \ar[r]^(0.6){\Phi\varepsilon} &\Phi_P\\
&&
}
\qquad
\xymatrix{{\Psi_P^\iota} \ar '[d]
'[drr] ^{id_{\mathcal{G}^P_\iota}}
[rr] 
\ar[r]^(0.35){\Psi\eta}& \Psi_P^\iota\Phi_P\Psi_P^\iota \ar[r]^(0.6){\varepsilon\Psi} &\Psi_P^\iota\\
&&
}
\end{equation}
which means that for each $(G,\omega)$ in $\mathcal{G}_P$ and each $\varphi$ in $\mathcal{G}^P$,
\begin{align}
1_{\Psi_P^\iota(\varphi)} &= \varepsilon_{\Psi_P^\iota(\varphi)}\circ \Psi_P^\iota(\eta_\varphi) \label{Psi_adj}\\
1_{\Phi_P((G,\omega))} &= \Phi_P(\varepsilon_{(G,\omega)})\circ\eta_{\Phi_P((G,\omega))} \label{Phi_adj}
\end{align}
We will start by verifying equation \ref{Psi_adj}. Let $\varphi\in\mathcal{G}^P_\iota$, we know that $\eta_\varphi:\varphi\longrightarrow\Phi_P\circ\Psi_P^\iota(\varphi)$ is a natural transformation defined for every $v\in P$ by 
\begin{equation}
\begin{array}{rrcl}
\eta_\varphi(v):&\varphi(v)&\longrightarrow &\Phi_P(\Psi_P^\iota(\varphi))(v)=\coprod_{u\le v}\varphi(u) \\
&x &\mapsto &(x,v)
\end{array}
\end{equation}
Then 
\[\Psi_P^\iota(\eta_\varphi):\Psi_P^\iota(\varphi)\to(\displaystyle\coprod_{v\in P} \Phi_P(\Psi_P^\iota(\varphi))(v),\omega^{ \Phi_P\Psi_P^\iota(\varphi)})=(\displaystyle\coprod_{v\in P} \coprod_{u\le v}\varphi(u),\omega^{ \Phi_P\Psi_P^\iota(\varphi)}),\] 
where $\omega^{ \Phi_P\Psi_P^\iota(\varphi)}|_{\coprod_{u\le v}\varphi(u)}=v$. From the definition of $\varepsilon$ we gave in Lemma \ref{eps} we deduce that 
\[\varepsilon_{\Psi_P^\iota(\varphi)}:\Psi_P^\iota\circ\Phi_P(\Psi_P^\iota(\varphi))\longrightarrow \Psi_P^\iota(\varphi).\] 

One has that $\Psi_P^\iota\circ\Phi_P(\Psi_P^\iota(\varphi))$ is the weighted graph $(\displaystyle\coprod_{v\in P}\Psi_P^\iota(\varphi)_v\;,\omega^{ \Phi_P\Psi_P^\iota(\varphi)})$, where $\Psi_P^\iota(\varphi)_v=\{x\in \Psi_P^\iota(\varphi) | \omega^\varphi(x)\le v\}=\displaystyle\coprod_{u\le v}\varphi(u)$, and $\omega^{ \Phi_P\Psi_P^\iota(\varphi)}|_{\coprod_{u\le v}\varphi(u)}=v$. 
\begin{equation}
\begin{array}{rrcl}
\varepsilon_{\Psi_P^\iota(\varphi)}:&(\displaystyle\coprod_{v\in P}\displaystyle\coprod_{u\le v}\varphi(u),\omega^{ \Phi_P\Psi_P^\iota(\varphi)})&\longrightarrow &\Psi_P^\iota(\varphi)\\
&((x,u),v)& \mapsto& (x,u)
\end{array}
\end{equation}
where $\Psi_P^\iota(\varphi)=(\displaystyle\coprod_{v\in P}\varphi(v), \omega^\varphi)$, with $\omega^\varphi|_{\varphi(v)}=v$. 

Then $\varepsilon_{\Psi_P^\iota(\varphi)}\circ \Psi_P^\iota(\eta_\varphi)=1_{\Psi_P^\iota(\varphi)}$ as the following shows:
\begin{equation}\label{eps_ver}
\begin{array}{rcccl}
(\displaystyle\coprod_{v\in P}\varphi(v), \omega^\varphi)&\xrightarrow{\Psi(\eta_\varphi)}&(\displaystyle\coprod_{v\in P}\coprod_{u\le v}\varphi(u),\omega^{ \Phi_P\Psi_P^\iota(\varphi)})&\xrightarrow{\varepsilon_{\Psi_P^\iota(\varphi)}}&(\displaystyle\coprod_{v\in P}\varphi(v), \omega^\varphi)\\
(x,v) &\mapsto&((x,v),v)& \mapsto& (x,v)
\end{array}
\end{equation}
\noindent We verify now identity \ref{Phi_adj}. Consider $(G,\omega)\in\mathcal{G}_P$, we have that 
\[\eta_{\Phi_P((G,\omega))}:\Phi_P((G,\omega))\to \Phi_P\circ\Psi_P^\iota(\Phi((G,\omega)))\] 
where $\Phi_P((G,\omega))(v)=G_v$, with $G_v=\{x\in G|\omega(x)\le v\}$. For every $v\in P$, we find that $\eta_{\Phi_P((G,\omega))}$ is determined by:
\begin{equation}
\begin{array}{rrcl}
\eta_{\Phi_P((G,\omega))}(v):&G_v &\longrightarrow &\displaystyle\coprod_{u\le v}G_u\\
&x &\mapsto &(x,v)
\end{array}
\end{equation}

 Considering that $\varepsilon_{(G,\omega)}:(\displaystyle\coprod_{v\in P} G_v,\omega^{\varphi_G}) \mapsto (G,\omega)$, where $\omega^{\varphi_G}|_{G_v}=v$. We have that $\Phi_P(\varepsilon_{(G,\omega)})$ is defined for every $v\in P$:
 \begin{equation}
 \begin{array}{rrcl}
 \Phi_P(\varepsilon_{(G,\omega)})(v):&\Phi_P(\Psi_P^\iota\circ\Phi_P((G,\omega)))(v)&\longrightarrow&\Phi_P((G,\omega))(v)\\
 \end{array}
 \end{equation}
 where 
 \[\Phi_P(\Psi_P^\iota\circ\Phi_P((G,\omega)))(v)=\{(x,u)\in \displaystyle\coprod_{v\in P} G_v \;\mbox{s.t.}\; \omega^{\varphi_G}((x,u))=u\le v\}=\displaystyle\coprod_{u\le v} G_u\] and $\Phi_P((G,\omega))(v)=G_v.$ 
 This gives the following natural transformation:
 
\begin{equation}
\begin{array}{rcccl}
G_v &\xrightarrow{\eta_{\Phi_P((G,\omega))}(v)}&\displaystyle\coprod_{u\le v}G_u &\xrightarrow{\Phi_P(\varepsilon_{(G,\omega)})(v)}&G_v\\
x &\xmapsto{\phantom{(G,A)}}&(x,v)&\xmapsto{\phantom{(G,A)}}& x
\end{array}
\end{equation} 
which proves that $\Phi_P(\varepsilon_{(G,\omega)})\circ\eta_{\Phi_P((G,\omega))}=1_{\Phi_P((G,\omega))}$.
\end{proof}

\subsection{Conclusions and application to homology}

\begin{definition}
Let $\tau\in\T^{P}$, the composition $\HHS_i\circ\tau\in\mathrm{Vect}_k^P$ will be called the $i^{\mathrm{th}}$ $P$-persistent homology of $\tau\in\T^P$.
\end{definition}

\noindent We have the following result which states that for any $P-$persistent object on the category of finite topological space, there is a $P-$weighted graph having the same $P-$persistent homology.

\begin{proposition}\label{topograph}
Let $\tau\in\T^{P},$ then there is $\theta\in\bar{\SG}^P $ such that
\begin{equation}
\HH_i^S\circ\tau\cong\HH_i\circ\cl\circ\theta
\end{equation}
as functors.
\end{proposition}

\begin{proof}
The commutativity of (\ref{****})  implies that the following diagram is commutative:
\begin{equation}\label{persistent}
\xymatrix{
\T^P\ar[r]&({\Tz})^P \ar[r]^{{(\HHS_i)}^P} \ar[d]_{{( k_1\circ\Ord)}^P}&{\mathrm{Vect}_k^P}\\
&\SG^P \ar @{->} @< 0pt>[ur]_{{(\HH_i\circ\cl)}^P}
}
\end{equation}
Therefore the statement holds with $\theta= k_1\circ\Ord\circ\tau.$
\end{proof}
The above result implies that $P$-persistent homology of finite spaces can be computed as $P$-persistent homology of graphs. We have then our main result on $P-$persistent homology.

\begin{theorem}\label{onegraph}
Let $\tau\in\T^P$ be a $P$-filtration of topological spaces such that $\tau_{ab}: X_a\to X_b$ is injective for all $a,b\in P$ with $a\leq b,$\\
Then exists a weighted graph $(G,\omega)\in\bar{\SG}_P$ such that $\HH_i^S\circ\tau\cong\HH_i(G,\omega)$.
\end{theorem}
\begin{proof}
It follows from Prop.\ref{topograph} and Th.\ref{isocat}.
\end{proof}

\begin{remark}
Let us consider a metric space $(X,d)$, since $X$ is $T_1$, then it has discrete topology. Therefore $\HHS_0$ is equal to $k^X$, and $\HHS_i=0$ for all $i>0$.\\ 
It is customary to associate to $(X,d)$ a nice simplicial complex, namely the Vietoris-Rips complex $X_\epsilon$. which is explicitly defined as $\cl(D_\epsilon(X))$ where $D_\epsilon$ is the graph with $V=X$, and $(x,y)\in E$ if and only if $d(x,y)\leq\epsilon$.\\
Therefore this approach is already included in our analysis.
\end{remark}



\section{Appendix on Homotopy and homology}\label{homsec}
\subsection{Homology}
The basic idea behind algebraic topology is to functorially attach algebraic objects to topological spaces in order to discern their properties. 
Homology theory does so by introducing functors from the category of topological spaces (or some related category) and continuous maps to the category of modules over a commutative base ring, such that these modules are topological invariants.\\
We will first introduce homology over simplicial complexes, which are our main setting, and then we will proceed to define it over general topological spaces.
\subsection{Simplicial homology}
Fixed a field $k$, in the following, by vector space we intend a $k-$vector space. \\
Given a simplicial complex $\Sigma$ of dimension $d$,  for $0\leq n \leq d$ consider the vector spaces $\CC_n:=\CC_n(\Sigma)$ with basis the set of $n$-faces in $\Sigma$. Elements in $\CC_n$ are called $n$\textit{-chains}.\\

The linear maps sending a $n$-face to the alternate sum of it's $(n-1)$-faces are called \textit{boundaries} and share the property $\partial_{n-1}\circ\partial_n=0$.

\begin{eqnarray*}
\partial_n: \CC_n & \longrightarrow & \CC_{n-1}\\
\text{[}p_{0},\ldots, p_{n}\text{]} & \rightarrow & \sum_{i=0}^n(-1)^i\text{[}p_{0},
\ldots, p_{i-1}, p_{i+1},\ldots, p_{n}\text{]}.
\end{eqnarray*}

The subspace $\ker\partial_n$ of $\CC_n$ is called the vector space of $n$-\textit{cycles} and denoted by $\ZZ_n:=\ZZ_n(\Sigma)$.
The subspace  $\mathrm{Im}\,\partial_{n+1}$ of $\CC_n$, is called the vector space of $n$-\textit{boundaries} and denoted by $\BB_n:=\BB_n(\Sigma)$.
\begin{remark}
From $\partial_{n-1}\circ\partial_n=0$ it follows that $\BB_n\subseteq \ZZ_n$ for all $n$.
\end{remark}
\begin{definition}
The $n-$th simplicial homology space of $\Sigma$, with coefficients in $k$, is the vector space $\HH_n:=\HH_n(\Sigma):=\ZZ_n/\BB_n$.
We denote by $\beta_n:=\beta_n(\Sigma)$ the rank of  $\HH_n:$ it is usually called the $n$-th Betti number of $\Sigma$. 
\end{definition}

The first Betti numbers of $\Sigma$ have an easy intuitive meaning: the $0$-th Betti number is the number of connected components of $\Sigma$, the first Betti number is the number of two dimensional (poligonal) holes, the third Betti number is the number of three dimensional holes (convex polyhedron).

\begin{remark}
It easy to check that $\CC_n,\ZZ_n,\BB_n$ and, therefore, $\HH_n$ are all functors $\sss\to \mathrm{Vect}_k$, where $\mathrm{Vect}_k$ denotes the category of vector spaces and linear mappings.
\end{remark}

There is plenty of literature on homology and in particular on simplicial homology, we refer the interested reader to \cite{Simp}. In particular, one can find thereby, the proof of the following.
\begin{proposition}
The functors $\HH_i$ are invariants by homeomorphism and homotopy type.

\end{proposition}
%
\begin{definition}
Let $G\in \SG$ be a graph. We now define as the homology space of $G$,  \[H_i(G):=H_i(\cl(G))\].
\end{definition}
\begin{proposition}
Let $\Sigma$ be a simplicial complex. Then, there exists a graph $G\in\SG$ such that $H_i(\Sigma)=H_i(G)$.
\end{proposition}
\begin{proof}
The proof is a consequence of Remark \ref{alexS} and \ref{**}. It is sufficient to consider as $G= k_1(\Ord(\pi(\Sigma)))$, the 1-skeleton of the barycentric subdivision of $\Sigma$, which is a { flag} complex.
\end{proof}
\subsection{Singular homology}
Simplicial homology has an analogous for general topological spaces, namely \textit{singular homology}, whose definition and properties we briefly recall now. Although we confine ourself into the category of finite topological spaces, the following definition remains valid for arbitrary topological spaces.
We address the interested reader to \cite{Hatcher, Simp} for a thorough treatise on these topics.

Let $X\in\T$ be a topological space, the chain spaces $\CC_n$ are in this case replaced by the vector spaces $\CC_n^S$ freely generated by the set of all continuous functions from the geometric realization of the standard n-simplex $\Delta^n$ to $X.$ \\
The boundaries are then defined in the following way thus making $(C_n^S,\partial_n^S)$ a chain complex.\\
\begin{definition}
Let $\sigma$ be a generator of $\CC_n$, i.e. a continuous function from $\Delta^n\to X$. Then the \textit{boundary homomorphism} $\partial_n^S$ can be constructed in the following way:
\[
\partial_n^S(\sigma)=\sum_i^n \sigma|_{[v_0,\dots,v_{i-1},v_{i+1},\dots,v_n]}
\]
where $\sigma|_{[v_0,\dots,v_{i-1},v_{i+1},\dots,v_n]}$ is the restriction of $\sigma$ to 
$[v_0,\dots,v_{i-1},v_{i+1},\dots,v_n].$
\end{definition}
It is easy to verify that $\partial_n^S\circ\partial_{n+1}^S=0$, thus we can define the homology spaces as we did for simplicial homology. We will denote the $i^{th}$ \textit{singular homology space} by $\HHS_i(X)$. For general nonsense it is easy to check that $H_i^S$ gives a functor $\T\to \mathrm{Vect}_k$.

\begin{theorem}[Theorem 2.27 in \cite{Hatcher}]\label{hom}
For any simplicial complex $\Sigma$, the singular homology groups are isomorphic to the simplicial homology groups.
\[
\forall i\in\N\qquad\HHS_i(\Sigma)\cong\HH_i(\Sigma)
\]
\end{theorem}

\begin{definition}
Let $X,Y\in\T$, and let $\pi_n(X,x)$ denote the homotopy group of the space $X$ at base point $x\in X$.\\
A map $f:X\to Y$ is a \textit{weak homotopy equivalence} if the following condition are both verified:
\begin{enumerate}
\item $f$ induces an isomorphism of the connected components of $X$ and $Y$
\[\Pi_0(f):\Pi_0(X)\to \Pi_0(Y)\]
\item for all $x\in X$, and $n\ge 1$ is an isomorphism on the homotopy groups 
\[\pi_n(f):\pi_n(X,x)\to \pi_n(Y,f(x))\]
\end{enumerate}
\end{definition}

There is the following result.
\begin{theorem}[McCord, \cite{mccord}]\label{mccord} 
Let $X\in\T$ with $X/\sim$ its Kolmogorov quotient, then $\mathcal{O}(X/\sim)\in\mathcal{F}$ is weak homotopy equivalent to $X.$
\end{theorem} 
We refer the interested reader to Chapter 1.4, \cite{Barmak}. In view of this result it makes sense to set $\Ord(X):=\Ord(X/\sim)$ for all $X\in\T.$

Theorem \ref{hom} and \ref{mccord} together imply that $\HHS_i(X)\cong \HH_i(\Ord(X))$ for all $X\in \T$.

Moreover, since $\Ord(X)$ is a { flag} complex $\HH_i(\Ord(X))=\HH_i(\cl( k_1(\Ord(X))))$, that is the graph homology of the graph which is the $1$-skeleton of $\Ord(X)$. Thus we can restrict ourselves to the study of the graph homology of $ k_1(\Ord(X))$.\\




In conclusion, the following diagram commutes:
\begin{equation} \label{****}
\xymatrix{\T\ar@{->}@< 2pt>[r]&\Tz \ar @{->} @< 2pt> [r]^{\HHS_i} \ar @{->} @< 0pt>[d]_\Ord &\mathrm{Vect}_k \\
&\mathcal{F} \ar @{->} @< 2pt> [r]^{k_1}
\ar @{<-} @<-2pt> [r] _\cl \ar @{->} [ur]^{\HH_i}&\SG \ar @{->} @< 1pt>[u]_{\HH_i}
}
\end{equation}


Although these observations are interesting per se, they become much more significant if we consider not only the homological structure of a data space but also its $P$-persistent properties. \\


\bibliographystyle{plain}

\begin{flushleft}
Francesco~Vaccarino\\
Dipartimento di Scienze Matematiche \\
Politecnico di Torino\\
C.so Duca degli Abruzzi n.24, Torino\\
10129, \ ITALIA \\
e-mail: \texttt{francesco.vaccarino@polito.it}\\
and\\
ISI Foundation\\
Via Alassio 11/c\\
10126 Torino - Italy\\
e-mail: \texttt{vaccarino@isi.it}\\
\bigskip

Alice~Patania\\
ISI Foundation\\
Via Alassio 11/c\\
10126 Torino - Italy\\
e-mail: \texttt{alice.patania@isi.it}\\
and\\
Dipartimento di Scienze Matematiche\\
Politecnico di Torino\\
C.so Duca degli Abruzzi n.24, Torino\\
10129, \ ITALIA \\
e-mail: \texttt{alice.patania@polito.it}\\

\bigskip
Giovanni~Petri\\
ISI Foundation\\
Via Alassio 11/c\\
10126 Torino - Italy\\
e-mail: \texttt{giovanni.petri@isi.it}

\end{flushleft}
\end{document}